\documentclass{amsart}

\usepackage{latexsym,enumerate}
\usepackage{amsmath,amsthm,amsopn,amstext,amscd,amsfonts,amssymb}
\usepackage[ansinew]{inputenc}
\usepackage{verbatim}
\usepackage{wasysym}
\usepackage[all]{xy}
\usepackage{epsfig} 
\usepackage{graphicx,psfrag} 
\usepackage{subfigure}
\usepackage{pstricks,pst-node}

\usepackage{multicol}
\usepackage{color}
\usepackage{colortbl}

\renewcommand{\SS}{\mathbb{S}^1}
\renewcommand{\L}{\mathcal{L}}

\newtheorem{theorem}{Theorem}[section]
\newtheorem{lemma}[theorem]{Lemma}
\newtheorem{proposition}[theorem]{Proposition}
\newtheorem{corollary}[theorem]{Corollary}

\newtheorem{example}[theorem]{Example}
\newtheorem{remark}[theorem]{Remark}

\DeclareMathOperator{\diam}{diam}

\newcommand{\spb}[1]{\smallskip}
\newcommand{\mpb}[1]{\medskip}
\newcommand{\bpb}[1]{\bigskip}

\newcommand{\e}{\varepsilon}
\renewcommand{\d}{\delta}

\newcommand{\g}{\gamma}
\newcommand{\G}{\Gamma}
\renewcommand{\th}{\theta}

\begin{document}
\DeclareGraphicsExtensions{.jpg,.pdf,.mps,.png}
\title{On the hyperbolicity constant of circular-arc graphs}

\author[Rosal{\'\i}o Reyes]{Rosal{\'\i}o Reyes}
\address{Facultad de Matem\'aticas, Universidad Aut\'onoma de Guerrero,
Carlos E. Adame No.54 Col. Garita, 39650 Acalpulco Gro., Mexico}
\email{khanclawn@hotmail.com}

\author[Jos\'e M. Rodr{\'\i}guez]{Jos\'e M. Rodr{\'\i}guez$^{(1)(2)}$}
\address{Departamento de Matem\'aticas, Universidad Carlos III de Madrid,
Avenida de la Universidad 30, 28911 Legan\'es, Madrid, Spain}
\email{jomaro@math.uc3m.es}
\thanks{$^{(1)}$ Supported in part by three grants from Ministerio de Econom{\'\i}a y Competititvidad, Agencia Estatal de
Investigación (AEI) and Fondo Europeo de Desarrollo Regional (FEDER) (MTM2013-46374-P, MTM2016-78227-C2-1-P and MTM2015-69323-REDT), Spain.}

\author[Jos\'e M. Sigarreta]{Jos\'e M. Sigarreta$^{(1)(2)}$}
\address{Facultad de Matem\'aticas, Universidad Aut\'onoma de Guerrero,
Carlos E. Adame No.54 Col. Garita, 39650 Acalpulco Gro., Mexico}
\email{josemariasigarretaalmira@hotmail.com}
\thanks{$^{(2)}$ Supported in part by a grant from CONACYT (FOMIX-CONACyT-UAGro 249818), M\'exico.}

\author[Mar{\'\i}a Villeta]{Mar{\'\i}a Villeta$^{(1)}$}
\address{Departamento de Estad{\'\i}stica e Investigaci\'on Operativa III, Facultad de Estudios Estad\'{\i}sticos, Universidad Complutense de Madrid,
Av. Puerta de Hierro s/n., 28040 Madrid, Spain}
\email{mvilleta@estad.ucm.es}

\date{\today}

\maketitle{}


\begin{abstract}
Gromov hyperbolicity is an interesting geometric property,
and so it is natural to study it in the context of geometric graphs.
It measures the tree-likeness of a graph from a metric viewpoint.
In particular, we are interested in circular-arc graphs, which is an important class of geometric intersection graphs.
In this paper we give sharp bounds for the hyperbolicity constant of (finite and infinite) circular-arc graphs.
Moreover, we obtain bounds for the hyperbolicity constant of the complement and line of any circular-arc graph.
In order to do that, we obtain new results about regular, chordal and line graphs which are interesting by themselves.
\end{abstract}


{\it AMS Subject Classification numbers $2010$:} 05C62; 05C63; 05C10; 05C75; 05C12.

\section{Introduction}

The concept of Gromov hyperbolicity grasps the essence of negatively curved
spaces like the classical hyperbolic space, simply connected Riemannian manifolds of
negative sectional curvature bounded away from $0$, and of discrete spaces like trees
and the Cayley graphs of many finitely generated groups (see the precise definition of hyperbolicity in the next section).
It is remarkable that a simple concept leads to such a rich
general theory (see \cite{GH, G1}). The first works on Gromov hyperbolic spaces deal with
finitely generated groups (see \cite{G1}). 
Initially, Gromov hyperbolic spaces were applied to the study of automatic groups in the science of computation
(see, e.g., \cite{O}); indeed, hyperbolic groups are strongly geodesically automatic, i.e., there is an automatic structure on the group \cite{Cha}. This concept appears also in algorithms
and networking.
For example, it has been shown empirically
in \cite{ShTa} that the internet topology embeds with better accuracy
into a hyperbolic space than into an Euclidean space
of comparable dimension (formal proofs that the distortion is related to the hyperbolicity can be found in \cite{VeSu});
the same holds for many complex networks, see \cite{KPKVB}.
A few algorithmic problems in
hyperbolic spaces and hyperbolic graphs have been considered
in recent papers (see \cite{Kra} and the references therein).
Furthermore, hyperbolic spaces are useful in secure transmission of information on the
network (see \cite{
K21}).

In \cite{T} it was proved the
equivalence of the hyperbolicity of many negatively curved surfaces
and the hyperbolicity of a graph related to it; hence, it is useful
to know hyperbolicity criteria for graphs from a geometrical viewpoint.
In recent years, the study
of mathematical properties of Gromov hyperbolic spaces
has become a topic of increasing interest in
graph theory and its applications (see, e.g., \cite{
BRRS,
BPK,BHB1,CDEHV,
CoCoLa,CD,CoDu,FIV,K50,
K21,
KoMo,KPKVB,MP,
Sha1,S,S2,Si,T,WZ} and the references therein).

For a finite graph with $n$ vertices it is possible to compute $\d(G)$ in time $O(n^{3.69})$ \cite{FIV} (this is improved in \cite{CoCoLa,CD}).
Given a Cayley graph (of a presentation with solvable word problem) there is an algorithm which allows to decide if it is hyperbolic \cite{Pap}.
A refinement of this approach has been proposed in \cite{ChChPaPe}, that allows to do the same for many graphs: in particular, it provides a simple constant-factor approximation of the hyperbolicity constant of a graph on $n$ vertices in $O(n^2)$
time when the graph is given by its distance-matrix.
However, deciding whether or not a general infinite graph is hyperbolic is usually very difficult.
Therefore, it is interesting to relate hyperbolicity with other properties of graphs.
The papers \cite{BCD,BHB1,WZ}
prove that chordal and $k$-chordal graphs are hyperbolic.
These results relating chordality and hyperbolicity are improved in \cite{BCRS,MP}.
Some other authors have obtained results on hyperbolicity for particular classes of graphs: vertex-symmetric graphs,
bipartite and intersection graphs, bridged graphs, expanders, and interval graphs
\cite{CaFu,CoDu,KoMo,LiTu,RS}.
In the recent paper \cite{CoDu2} new symmetric graph classes are proved to be non-hyperbolic.

A geometric graph is a graph in which the vertices or edges are associated with geometric objects.
One of the main classes of geometric graphs are intersection graphs.
An \emph{intersection graph} is a graph in which each vertex is associated with a set and in which vertices are connected by edges
whenever the corresponding sets have a nonempty intersection.
Since hyperbolicity is an interesting geometric property, it is natural to study this property for some classes of geometric graphs.
Intersection graphs are very important in both theoretical and application
point of view (see, e.g., \cite{60}).
In \cite{RS} the authors study the hyperbolicity constant of interval graphs (an interval graph is the intersection graph of a family of intervals on the real line).
In this paper we work with circular-arc graphs (another important class of intersection graphs).

A \emph{circular-arc graph} (or \emph{circular graph}) is the intersection graph of a family of arcs on the unit circle $\SS$.
It has one vertex for each arc in the family, and an edge between every pair of corresponding vertices to arcs that intersect.
Circular-arc graphs are useful in modeling periodic resource allocation problems in operations research (each arc represents a request for a resource for a specific period repeated in time).
They also have applications in different fields such as genetic
research, traffic control, computer compiler design and statistics (see, e.g., \cite{Pal}).
In \cite{65} appears an $O(n + m)$ time algorithm for recognizing a circular-arc graph (with $n$ vertices and $m$ edges).

Of course, every interval graph can be viewed as a circular-arc graph;
if a representation of a circular-arc graph $G$ leaves some point of the unit circle uncovered, it is topologically the same as an interval representation of $G$
(by cutting the circle and straighten it out to a straight line);
we will use this identification along the paper.

In this paper we give sharp bounds for the hyperbolicity constant of circular-arc graphs (see Theorem \ref{t:main}).
These bounds are improved in Theorem \ref{t:main2} for proper circular-arc graphs.
Theorem \ref{t:l4} gives a sufficient condition in order to attain the lower bound of $\d(G)$ in Theorem \ref{t:main};
in particular, it shows that this bound is sharp.
Propositions \ref{t:0} and \ref{t:34} characterize the circular-arc graphs with small hyperbolicity constant.
In Section 4, we obtain bounds in Theorems \ref{t:comple} and \ref{t:line} for the hyperbolicity constant of the complement and line of a circular-arc graph, respectively.
These theorems improve, for circular-arc graphs, the general bounds for the hyperbolicity constant of the complement and line graphs.
In order to do that, we obtain new results about regular, chordal and line graphs which are interesting by themselves (see Theorems \ref{t:regular} and \ref{t:chordal}).
Although some results of this paper can be viewed as generalizations to the context of circular-arc graphs of results for interval graphs (see \cite{RS}),
we want to remark that Theorem \ref{t:line} is new even for interval graphs.

Besides \cite{RS}, in \cite{CoDu} sharp inequalities for the hyperbolicity constant of several classes of intersection graphs are provided. In particular, for
line, clique and biclique graphs, and some extensions of line graphs (incidence, total, middle and $k$-edge graphs).
Hence, the results in \cite{CoDu} are different from the results in this paper.
Furthermore, the bounds in \cite{CoDu} refer to the hyperbolicity constant with respect to the four-point definition of hyperbolicity;
the inequalities for the hyperbolicity constant with respect to a definition can be translated to the hyperbolicity constant with respect to another definition, with additional multiplicative and/or additive constants (for instance, multiplying or dividing by $3$ the initial upper or lower bound, respectively).
But note that it is not difficult to check that circular-arc graphs are hyperbolic;
for a fixed definition of hyperbolicity, the challenge is to obtain sharp bounds for the hyperbolicity constant.

\section{Background and previous results}

We collect in this section some previous definitions and results which will be useful along the paper.

\smallskip

We say that the curve $\g$ in a metric space $X$ is a
\emph{geodesic} if we have $L(\g|_{[t,s]})=d(\g(t),\g(s))=|t-s|$ for every $s,t\in [a,b]$
(then $\gamma$ is equipped with an arc-length parametrization).
The metric space $X$ is said \emph{geodesic} if for every couple of points in
$X$ there exists a geodesic joining them; we denote by $[xy]$
any geodesic joining $x$ and $y$; this notation is ambiguous, since in general we do not have uniqueness of
geodesics, but it is very convenient.
Consequently, any geodesic metric space is connected.
If the metric space $X$ is
a graph, then the edge joining the vertices $u$ and $v$ will be denoted by $uv$.

Throughout this paper, $G=(V,E)=(V(G),E(G))$ denotes a (finite or infinite) simple (without loops and multiple edges)
graph (not necessarily connected) such that $V\neq \emptyset$ and
every edge has length $1$.
In order to consider a graph $G$ as a geodesic metric space, identify (by an isometry)
any edge $uv\in E(G)$ with the interval $[0,1]$ in the real line;
then the edge $uv$ (considered as a graph with just one edge)
is isometric to the interval $[0,1]$.
Thus, the points in $G$ are the vertices and, also, the points in the interior
of any edge of $G$.
In this way, any connected graph $G$ has a natural distance
defined on its points, induced by taking shortest paths in $G$,
and we can see $G$ as a metric graph.
We denote by $d_G$ or $d$ this distance.
If $x,y$ are in different connected components of $G$, we define $d_G(x,y)=\infty$.
These properties guarantee that any connected component of any graph is a geodesic metric space.

If $X$ is a geodesic metric space and $x_1,x_2,x_3\in X$, the union
of three geodesics $[x_1 x_2]$, $[x_2 x_3]$ and $[x_3 x_1]$ is a
\emph{geodesic triangle} that will be denoted by $T=\{x_1,x_2,x_3\}$
and we will say that $x_1,x_2$ and $x_3$ are the vertices of $T$; it
is usual to write also $T=\{[x_1x_2], [x_2x_3], [x_3x_1]\}$. We say
that $T$ is $\d$-{\it thin} if any side of $T$ is contained in the
$\d$-neighborhood of the union of the two other sides. We denote by
$\d(T)$ the sharp thin constant of $T$, i.e., $ \d(T):=\inf\{\d\ge 0\,|\,
\, T \, \text{ is $\d$-thin}\}. $ The space $X$ is
$\d$-\emph{hyperbolic} $($or satisfies the {\it Rips condition} with
constant $\d)$ if every geodesic triangle in $X$ is $\d$-thin.
If we have a triangle with two
identical vertices, we call it a ``bigon"$\!$. Obviously, every bigon in
a $\d$-hyperbolic space is $\d$-thin.
We denote by $\d(X)$ the sharp hyperbolicity constant of $X$, i.e.,
$\d(X):=\sup\{\d(T)\,|\, \, T \, \text{ is a geodesic triangle in }X \}.$
We say that $X$ is \emph{hyperbolic} if $X$ is $\d$-hyperbolic for some $\d \ge 0$; then $X$ is hyperbolic if and only if $ \d(X)<\infty.$
If $X$ has connected components $\{X_i\}_{i\in I}$, then we define $\d(X):=\sup_{i\in I} \d(X_i)$, and we say that $X$ is hyperbolic if $\d(X)<\infty$.

In the classical references on this subject (see, e.g., \cite{GH}) appear
several different definitions of Gromov hyperbolicity, which are equivalent in
the sense that if $X$ is $\d$-hyperbolic with respect to one definition, then it is
$\d'$-hyperbolic with respect to another definition (for some $\d'$ related to $\d$).
We have chosen this definition by its deep geometric meaning \cite{GH}.

We want to remark that the main examples of hyperbolic graphs are the trees.
In fact, the hyperbolicity constant of a geodesic metric space can be viewed as a measure of
how ``tree-like'' the space is, since those spaces $X$ with $\delta(X) = 0$ are precisely the metric trees.
This is an interesting subject since, in
many applications, one finds that the borderline between tractable and intractable
cases may be the tree-like degree of the structure to be dealt with
(see, e.g., \cite{CYY}).
However, the hyperbolicity constant does not relate the graph in question to a specific tree (if connected) or forest (if not connected). In \cite{Sha5}, a measure called forest likelihood is introduced to estimate the likelihood of any given forest via a random dynamical generation process. This measure establishes an interesting connection between static graphs and dynamically growing graphs.

\smallskip

For any graph $G$, we define, as usual,
$$
\begin{aligned}
\diam V(G) & := \sup \big\{ d_G(v,w)\,|\,\, v,w \in V(G) \big\},
\\
\diam G & := \sup \big\{ d_G(x,y)\,|\,\, x,y \in G \big\}.
\end{aligned}
$$
i.e, $\diam V(G)$ is the diameter of the set of vertices of $G$, and $\diam G$ is the diameter of the whole graph $G$
(recall that in order to have a geodesic metric space, $G$ must contain both the vertices and the points in the interior of any edge of $G$).

\smallskip

The following result is well-known (see, e.g., \cite[Theorem 8]{RSVV} for a proof).

\begin{lemma} \label{dddd}
In any graph $G$ the inequality $\delta(G) \leq \frac{1}{2} \diam G$
holds.
\end{lemma}

We say that a subgraph $\G$ of $G$ is \emph{isometric} if $d_{\G}(x,y)=d_{G}(x,y)$ for every $x,y\in \G.$

\smallskip

We need the following elementary results (see, e.g., \cite[Lemma 5 and Theorem 11]{RSVV} for some proofs).

\begin{lemma} \label{lem_LX}
If $\G$ is an isometric subgraph of $G$, then $\d(\G) \le \d(G)$.
\end{lemma}

\begin{lemma} \label{cyclegarph}
If $C_n$ denotes the cycle graph with $n\ge 3$ vertices, then $\delta(C_n) = n/4$.
\end{lemma}

If $G$ is a circular-arc graph, then a set of vertices $K=\{v_1,\dots,v_r\}$ and corresponding arcs $\{I_1,\dots,I_r\}$ is said
\emph{total} if $I_1\cup \cdots \cup I_r = \SS$, and we say that $r$ is the \emph{size} of $K$.

We say that a circular-arc graph $G$ is $NI$ if it has a total set of vertices.
If either $G$ is a finite circular-arc graph or every arc is open, then $G$ is $NI$ if and only if the union of the corresponding arcs to vertices in $G$ is $\SS$.
Note that a circular-arc graph $G$ is also an interval graph if and only if it is not NI.
In \cite{RS} the authors study the hyperbolicity constant of interval graphs.

For any $NI$ circular-arc graph $G$, let us define
$$
\varrho(G) := \min \big\{ \, size(K)\, |\; K \, \text{ is a total set of vertices in }\, G\,\big\}.
$$
If $G$ is an interval graph, then we define $\varrho(G) := 0$.
Hence, a circular-arc graph $G$ is NI if and only if $\varrho(G) \ge 1$.
Note that $\varrho = 1$ if and only if an arc is the whole unit circle $\SS$.

\smallskip

As usual, by \emph{cycle} we mean a simple closed curve, i.e., a path with different vertices,
unless the last one, which is equal to the first vertex.

Given a graph $G$, we denote by
$J(G)$ the union of the set $V(G)$ and the midpoints of the edges
of $G$. Consider the set $\mathbb{T}_1$ of geodesic triangles $T$ in
$G$ that are cycles and
such that the three vertices of the triangle $T$ belong to $J(G)$, and
denote by $\delta_1(G)$ the infimum of the constants
$\lambda$ such that every triangle in $\mathbb{T}_1$ is
$\lambda$-thin.

The following result, which appears in \cite[Theorems 2.5,
2.6 and 2.7]{BRS}, will be used throughout the paper.

\begin{theorem} \label{t:BRS}
For every graph $G$ we have $\delta_1(G)=\delta(G)$.
Furthermore, if $G$ is hyperbolic, then $\delta(G)$ is a multiple of $1/4$
and there exists $T \in \mathbb{T}_1$ with $\delta(T)=\delta(G)$.
\end{theorem}

The following result in \cite[Theorem 11]{MRSV} will be useful.

\begin{theorem} \label{t:04}
If $G$ is a graph with edges of length $1$ with $\d(G)< 1$, then we have either $\d(G) = 0$ or $\d(G) = 3/4$.
Furthermore,

\begin{itemize}
\item $\d(G) = 0$ if and only if $G$ is a tree.

\item $\d(G) = 3/4$ if and only if $G$ is not a tree and every cycle in $G$ has length $3$.
\end{itemize}
\end{theorem}

We say that a vertex $v$ in a graph $G$ is a \emph{cut-vertex} if $G\setminus v$ is not connected.
A graph is \emph{biconnected} if it does not contain cut-vertices.
Given a graph $G$, we say that a family of subgraphs $\{G_{s} \}_s$ of $G$ is a \emph{T-decomposition} of $G$ if $\cup_s G_{s} = G $ and $G_{s}\cap G_{r}$ is either a \emph{cut-vertex} or the empty set for each $s\neq r$.
The well-known biconnected decomposition of any graph is an example of T-decomposition.

\smallskip

It is known that the hyperbolicity constant of a graph is the supremum of the hyperbolicity constants of its biconnected components \cite{G1}.
One can check that the following result also holds (see, e.g., \cite[Theorem 3]{BRSV2} for a proof).

\begin{proposition} \label{t:Tdec}
Let $G$ be a graph and $\{G_{s} \}_s$ be any T-decomposition of $G$, then
$$
\delta(G) =\sup_{s} \delta (G_{s}) .
$$
\end{proposition}

By \cite[Proposition 5 and Theorem 7]{MRSV}, we have the following result.

\begin{lemma} \label{p:c3}
If $G$ is any graph with a cycle $g$ with length $L(g) \ge 3$,
then $\d(G)\ge 3/4$.
If there exists a cycle $g$ in $G$ with length $L(g) \ge 4$,
then $\d(G)\ge 1$.
\end{lemma}

We recall some facts from \cite{RS}.

Let $G$ be an interval graph.

We say that $G$ has the $0$-\emph{intersection property} if for every three corresponding intervals
$I'$, $I''$ and $I'''$ to vertices in $G$ we have
$I'\cap I''\cap I''' = \emptyset$.

$G$ has the $(3/4)$-\emph{intersection property} if it does not have the $0$-intersection property
and for every four corresponding intervals
$I'$, $I''$, $I'''$ and $I''''$ to vertices in $G$ we have
$I'\cap I''\cap I''' = \emptyset$ or $I'\cap I''\cap I'''' = \emptyset$.

By a \emph{couple} of intervals in a cycle $C$ of $G$ we mean the union of two non-disjoint intervals whose corresponding vertices belong to $C$.
We say that $G$ has the $1$-\emph{intersection property} if it does not have the $0$ and $(3/4)$-intersection properties
and for every cycle $C$ in $G$ each interval and couple of corresponding intervals to vertices in $C$
are not disjoint.

Let $G$ be an interval graph.
We say that $G$ has the $(3/2)$-\emph{intersection property} if there exists two disjoint corresponding intervals
$I'$ and $I''$ to vertices in
a cycle $C$ in $G$ such that there is no interval $I$ (corresponding to a vertex in $G$) with
$I\cap I'\neq \emptyset$ and $I\cap I'' \neq \emptyset$.

\smallskip

\cite[Theorem 3.14]{RS} give the following result.

\begin{theorem} \label{034}
Every interval graph $G$ is hyperbolic and $\d(G) \in \{0,3/4,1,5/4,3/2\}$.
Furthermore,

\begin{itemize}
\item $\d(G) = 0$ if and only if $G$ has the $0$-intersection property.

\item $\d(G) = 3/4$ if and only if $G$ has the $(3/4)$-intersection property.

\item $\d(G) = 1$ if and only if $G$ has the $1$-intersection property.

\item $\d(G) = 5/4$ if and only if $G$ does not have the $0$, $3/4$, $1$ and $(3/2)$-intersection properties.

\item $\d(G) = 3/2$ if and only if $G$ has the $(3/2)$-intersection property.
\end{itemize}
\end{theorem}


It is well-known that the interval graphs $G$ with $\d(G) = 0$ are the caterpillar trees (the trees for which removing the leaves and incident edges produces a
path graph), see \cite{Ju}, but the characterization with the 0-intersection property is interesting, since it looks similar to the other intersection properties.

\smallskip

If $C$ is a cycle in $G$ and $v\in V(G)$, we denote by $\deg_C (v)$ the degree of the vertex $v$
in the subgraph $\Gamma$ induced by $V(C)$ (note that $\Gamma$ could contain edges that are not contained in $C$, and thus it is possible to have $\deg_C (v)>2$).

In \cite[Theorem 3.2]{BRRS} appears the following result.

\begin{theorem} \label{t:delta2}
Given any graph $G$, we have $\delta(G) \ge 5/4$ if and only if there exist a cycle
$g$ in $G$ with length $L(g) \ge 5$ and a vertex $w\in g$
such that $\deg_g(w)=2$.
\end{theorem}

The following result appears in \cite[Theorem 4.9]{HPR2}.

\begin{theorem} \label{t:HPR2}
If $G$ is a graph with $n$ vertices and minimum degree $n-3$, then $\d(G)\le 5/4$.
\end{theorem}

The following result in \cite[Theorem 2.2]{BRST} gives a sharp bound for the hyperbolicity constant of the complement of a graph.

\begin{theorem} \label{t:comple00}
If $G$ is a graph with $\diam(V(G)) \ge 3$, then its complement graph $\overline{G}$ satisfies $0 \le \d(\overline{G}) \le 2$.
\end{theorem}

Finally, we will need the following result in \cite[Theorem 4.8]{RS} that improves Theorem \ref{t:comple00} for interval graphs
(recall that the most difficult case in the study of the complement of a graph is the set of graphs $G$ with $\diam V(G) = 2$).

\begin{theorem} \label{t:comple0}
If $G$ is an interval graph, then $0 \le \d(\overline{G}) \le 2$.
\end{theorem}

\section{Circular-arc graphs and hyperbolicity}

The parameter $\varrho(G)$ plays an important role in the study of the hyperbolicity of circular-arc graphs, as the following result shows.
Recall that $\lfloor t \rfloor$ denotes the lower integer part of the real number $t$, i.e., the greatest integer least than or equal to $t$.

Since any NI circular-arc graph is a bounded set, we have that it is hyperbolic.
The following result provides sharp inequalities for the hyperbolicity constant of any circular-arc graph.

\begin{theorem} \label{t:main}
Let $G$ be a circular-arc graph.
If $\varrho(G)\neq 1,2,$ then $G$ satisfies the sharp inequalities
$$
\frac14\, \varrho(G) \le \d(G)\le \frac12 \Big\lfloor \frac12\, \varrho(G) \Big\rfloor + \, \frac32\, .
$$
If $\varrho(G) = 1$, then $G$ satisfies the sharp inequalities
$$
0 \le \d(G)\le \frac32\, .
$$
If $\varrho(G) = 2$, then $G$ satisfies the sharp inequalities
$$
0 \le \d(G)\le 2 .
$$
\end{theorem}

\begin{proof}
The result is known if $G$ is an interval graph (i.e., if $\varrho(G) =0$), see \cite[Corollary 4.1]{RS}.

Assume now that $\varrho(G) \ge 1$.
Let us prove the upper bound of $\d(G)$.
Fix any set of vertices $K=\{v_1,\dots,v_{\varrho(G)}\}$ and corresponding arcs $\{I_1,\dots,I_{\varrho(G)}\}$ with $I_1\cup \cdots \cup I_{\varrho(G)} = \SS$.
Thus, every arc in $\SS$ intersects some arc in $\{I_1,\dots,I_{\varrho(G)}\}$.
Hence,
$$
\begin{aligned}
\diam V(G)
& \le 1  + \diam K + 1 = \Big\lfloor \frac12\, \varrho(G) \Big\rfloor + 2,
\\
\diam G
& \le \frac12  + \diam V(G) + \frac12 \le \Big\lfloor \frac12\, \varrho(G) \Big\rfloor + 3,
\end{aligned}
$$
and Lemma \ref{dddd} gives the upper bound.

Let us prove now that this bound is sharp.
Given $\th_1< \th_2$, denote by $[e^{i\th_1},e^{i\th_2}]$ the arc
$$
[e^{i\th_1},e^{i\th_2}] := \big\{ e^{i\th}\, |\; \th \in [\th_1,\th_2]\big\}.
$$
Fix any even integer $\varrho \ge 6$ with $\varrho \equiv 2 \,(\!\!\!\mod 4)$
and consider the family of arcs
$$
\big\{ [e^{2\pi i(j-1)/ \varrho},e^{2\pi ij/ \varrho}]\big\}_{j=1}^{\varrho}
$$
Denote by $I_j$ the arc $ [e^{2\pi i(j-1)/ \varrho},e^{2\pi ij/ \varrho}] $.
Let $z_1,z_2,z_3$ be the points $e^{\pi i/ (2\varrho)},e^{2\pi i/ (2\varrho)},e^{3\pi i/ (2\varrho)}$ in $I_1$, respectively,
$z_4,z_5$ the points $e^{2\pi i/ \varrho+2\pi i/ (3\varrho)},e^{2\pi i/ \varrho+4\pi i/ (3\varrho)}$ in $I_2$, respectively,
and $z_j$ the midpoint of $I_{j-3}$ with $6 \leq j \leq \varrho /2 +3$.
Let $z_{\varrho /2 +4},z_{\varrho /2 +5},z_{\varrho /2 +6}$ be the points $-e^{\pi i/ (2\varrho)},-e^{2\pi i/ (2\varrho)},-e^{3\pi i/ (2\varrho)}$ in $I_{\varrho /2 + 1}$, respectively,
$z_{\varrho /2 +7},z_{\varrho /2 +8}$ the points $-e^{2\pi i/ \varrho+2\pi i/ (3\varrho)},-e^{2\pi i/ \varrho+4\pi i/ (3\varrho)}$ in $I_{\varrho /2 + 2}$, respectively,
and $z_k$ the midpoint of $I_{k-6}$ with $\varrho/2 + 9 \leq k \le \varrho + 6$.

Consider the circular-arc graph $G_\varrho$ defined as the intersection graph of the family of arcs
$$
\big\{ [e^{2\pi i(j-1)/ \varrho},e^{2\pi ij/ \varrho}]\big\}_{j=1}^{\varrho}
\cup \big\{ [z_j,z_{j+1}]\big\}_{j=1}^{\varrho+6}
\cup \big\{ [z_{\varrho+6},z_1]\big\}.
$$
Let $x$ (respectively, $y$) be the midpoint of the edge of $G_\varrho$ with endpoints corresponding to the arcs
$[z_1,z_2]$ and $[z_2,z_3]$
(respectively, $[z_{\varrho/2+4},z_{\varrho/2+5}]$ and $[z_{\varrho/2+5},z_{\varrho/2+6}]$).

We have $d_{G_\varrho}(x,y)= 3/2+ \varrho/2 + 3/2 = \varrho/2 +3$.
Let $\g_1$ and $\g_2$ be two geodesics in $G_\varrho$ joining $x$ and $y$ such that
$\g_1$ contains the corresponding vertices to the arcs
$\big\{ [z_j,z_{j+1}]\big\}_{j=2}^{\varrho/2+4}$
and $\g_2$ contains the corresponding vertices to the arcs
$\big\{ [z_j,z_{j+1}]\big\}_{j=\varrho/2+5}^{\varrho+6} \cup \big\{ [z_{\varrho+6},z_1]   \big\}  \cup \big\{ [z_1,z_2]   \big\}$.
Consider the geodesic bigon $\{\g_1,\g_2\}$.
If $p$ is the midpoint of $\g_1$, then
$d_{G_\varrho}(p,\g_2)= d_{G_\varrho}(x,y)/2=\varrho/4 + 3/2$.
Hence, $\varrho/4 + 3/2 = d_{G_\varrho}(p,\g_2)\le \d(G_\varrho) \le \varrho/4 + 3/2$, and we conclude $\d(G_\varrho) = \varrho/4 + 3/2$.

If $\varrho(G) = 1$, then this upper bound is also attained.
The wheel graph with seven vertices $W_7$ is a circular-arc graph with $\varrho(G)= 1$, and \cite[Theorem 11]{RSVV} gives that $\d(W_7) = 3/2$.

If $\varrho(G) = 2$, then this upper bound is attained by the circular-arc graph $G$ corresponding to the arcs
$$
\begin{aligned}
& [e^{0i},e^{\pi i}] \cup [e^{\pi i},e^{2\pi i}]
\cup [e^{0i},e^{\pi i/4}] \cup [e^{\pi i/4},e^{\pi i/2}]\cup [e^{\pi i/2},e^{3\pi i/4}] \cup [e^{3\pi i/4},e^{\pi i}]
\\
& \qquad \qquad \qquad \qquad \;\; \cup [e^{\pi i},e^{5\pi i/4}] \cup [e^{5\pi i/4},e^{3\pi i/2}]\cup [e^{3\pi i/2},e^{7\pi i/4}] \cup [e^{7\pi i/4},e^{2\pi i}].
\end{aligned}
$$
In order to prove $\d(G)=2$, let $x$ (respectively, $y$) be the midpoint of the edge in $G$
with endpoints $[e^{\pi i/4},e^{\pi i/2}]$ and $[e^{\pi i/2},e^{3\pi i/4}]$
(respectively, $[e^{5\pi i/4},e^{3\pi i/2}]$ and $[e^{3\pi i/2},e^{7\pi i/4}]$).
One can check that there are two geodesics $\g_1$ and $\g_2$ such that the midpoint $p$ of $\g_1$ satisfies $d_{G}(p,\g_2)=2$.
If we consider the geodesic bigon $\{\g_1,\g_2\}$, then $2 = d_{G}(p,\g_2)\le \d(G) \le 2$, and we conclude $\d(G) = 2$.

In order to prove the lower bound of $\d(G)$, we deal first with the case $\varrho(G)\ge 3$.
As  above, fix any set of vertices $K=\{v_1,\dots,v_{\varrho(G)}\}$ and corresponding arcs
$\{I_1,\dots,I_{\varrho(G)}\}$ with $I_1\cup \cdots \cup I_{\varrho(G)} = \SS$.
The definition of $\varrho(G)$ gives that the subgraph $\G_K$ of $G$ induced by $K$ is an isometric subgraph of $G$.
Since $\varrho(G)\ge 3$, the subgraph $\G_K$ is isomorphic to the cycle graph $C_{\varrho(G)}$.
Therefore, Lemmas \ref{lem_LX} and \ref{cyclegarph} give
$\d(G) \ge \d(\G_K) = \varrho(G)/4$.

Any circular-arc graph isomorphic to the cycle graph $C_\varrho$ attains this lower bound.

Finally, the lower bounds for the cases $\varrho(G) = 1, 2$ are trivial, and they are attained by the graphs $G_1$ with just a vertex and $G_2$ with just an edge, respectively.
\end{proof}

An important subset of circular-arc graphs are proper circular-arc graphs.
A circular-arc graph $G$ is said \emph{proper} if there is a representation of $G$ where none of the arcs contains another.
The following result improves Theorem \ref{t:main} for this kind of graphs.

\begin{theorem} \label{t:main2}
Let $G$ be a proper circular-arc graph.
If $\varrho(G)\geq 3$, then $G$ satisfies the sharp inequalities
$$
\frac14\, \varrho(G) \le \d(G)\le \frac12 \Big\lfloor \frac12\, \varrho(G) \Big\rfloor + 1 .
$$
If $\varrho(G) = 1$, then $\d(G)=0$.
If $\varrho(G) = 2$, then $G$ satisfies the sharp inequalities $0 \le \d(G)\le 5/4$.
\end{theorem}

\begin{proof}
Assume first that $\varrho(G) =1$.
Since $\SS$ is a corresponding arc to a vertex of $G$
and $G$ is a proper circular-arc graph, we have that it has just a vertex, and $\d(G)=0$.

Assume now that $\varrho(G) \ge 2$.
The lower bounds are a consequence of Theorem \ref{t:main}
(note that the examples in the proof of Theorem \ref{t:main} attaining the lower bounds are proper circular-arc graphs).
Let us prove the upper bound of $\d(G)$.
Fix any set of vertices $K=\{v_1,\dots,v_{\varrho(G)}\}$ and corresponding arcs $\{I_1,\dots,I_{\varrho(G)}\}$ with $I_1\cup \cdots \cup I_{\varrho(G)} = \SS$.
Thus, every arc in $\SS$ intersects two arcs in $\{I_1,\dots,I_{\varrho(G)}\}$.
Given any $u,w\in V(G) \setminus K$, there are $v_{1,u},v_{2,u},v_{1,w},v_{2,w} \in K$ with
$uv_{1,u}, uv_{2,u}, wv_{1,w}, wv_{2,w}, v_{1,u}v_{2,u}, v_{1,w}v_{2,w}\in E(G)$.
Thus,
$$
d_G(v_{1,u}v_{2,u}, v_{1,w}v_{2,w})
\le \Big\lfloor \frac12\, (\varrho(G)-2) \Big\rfloor = \Big\lfloor \frac12\, \varrho(G) \Big\rfloor -1.
$$
Hence,
$$
\begin{aligned}
\diam V(G)
& \le 1 + \Big( \Big\lfloor \frac12\, \varrho(G) \Big\rfloor -1\Big) + 1
= \Big\lfloor \frac12\, \varrho(G) \Big\rfloor + 1,
\\
\diam G
& \le \frac12  + \diam V(G) + \frac12 \le \Big\lfloor \frac12\, \varrho(G) \Big\rfloor + 2,
\end{aligned}
$$
and Lemma \ref{dddd} gives the upper bound for $\varrho(G)\geq 3$,
and $\d(G)\le 3/2$ if $\varrho(G) = 2$.

Let us prove now that the upper bound for $\varrho(G)\geq 3$ is sharp.
Fix any even integer $\varrho \ge 4$ and $0<\e < \pi/\varrho$,
and consider the proper circular-arc graph $\G_{\varrho}$ defined as the intersection graph of the family of arcs
$$
\begin{aligned}
\big\{ [e^{2\pi i(j-1)/\varrho},e^{2\pi ij/\varrho}]\big\}_{j=1}^{\varrho}
& \cup \big\{ [e^{\e i + 2\pi i(j-1)/\varrho},e^{\e i + 2\pi ij/\varrho}]\big\}_{j=1}^{\varrho/2-1}
\cup \big\{ [e^{\e i + \pi i + 2\pi i(j-1)/\varrho},e^{\e i + \pi i + 2\pi ij/\varrho}]\big\}_{j=1}^{\varrho/2-1}
\\
& \cup [e^{-\e i + \pi i-2\pi i/\varrho},e^{-\e i + \pi i}]
\cup [e^{-\e i + 2\pi i-2\pi i/\varrho},e^{-\e i + 2\pi i}]
\\
& \cup [e^{-\e i/2},e^{\e i}]
\cup [e^{-\e i/2 + \pi i},e^{\e i + \pi i}]
\\
& \cup [e^{-\e i + \pi i},e^{\e i/2 + \pi i}]
\cup [e^{-\e i + 2\pi i},e^{\e i/2 + 2\pi i}] .
\end{aligned}
$$
Let $x$ (respectively, $y$) be the midpoint of the edge of $\G_{\varrho}$ with endpoints corresponding to the arcs
$[e^{-\e i/2},e^{\e i}]$ and $[e^{-\e i + 2\pi i},e^{\e i/2 + 2\pi i}]$
(respectively, $[e^{-\e i/2 + \pi i},e^{\e i + \pi i}]$ and $[e^{-\e i + \pi i},e^{\e i/2 + \pi i}]$).
We have $d_{\G_{\varrho}}(x,y)=\varrho/2 + 2$.
Let $\g_1$ and $\g_2$ be two geodesics in $\G_{\varrho}$ joining $x$ and $y$ such that
$\g_1$ contains the corresponding vertices to the arcs
$$
[e^{-\e i/2},e^{\e i}]
\cup \big\{ [e^{\e i + 2\pi i(j-1)/\varrho},e^{\e i + 2\pi ij/\varrho}]\big\}_{j=1}^{\varrho/2-1}
\cup [e^{-\e i + \pi i-2\pi i/\varrho},e^{-\e i + \pi i}]
\cup [e^{-\e i + \pi i},e^{\e i/2 + \pi i}]
$$
and $\g_2$ contains the corresponding vertices to the arcs
$$
[e^{-\e i/2 + \pi i},e^{\e i + \pi i}]
\cup \big\{ [e^{\e i + \pi i + 2\pi i(j-1)/\varrho},e^{\e i + \pi i + 2\pi ij/\varrho}]\big\}_{j=1}^{\varrho/2-1}
\cup [e^{-\e i + 2\pi i-2\pi i/\varrho},e^{-\e i + 2\pi i}]
\cup [e^{-\e i + 2\pi i},e^{\e i/2 + 2\pi i}] .
$$
Consider the geodesic bigon $\{\g_1,\g_2\}$.
If $p$ is the midpoint of $\g_1$, then
$d_{\G_{\varrho}}(p,\g_2)= d_{\G_{\varrho}}(x,y)/2=\varrho/4 + 1$.
Hence, $\varrho/4 + 1 = d_{\G_{\varrho}}(p,\g_2)\le \d(\G_{\varrho}) \le \varrho/4 + 1$, and we conclude $\d(\G_{\varrho}) = \varrho/4 + 1$.

Assume that $\varrho(G) = 2$.
We have proved $\d(G)\le 3/2$.
Seeking for a contradiction assume that $\d(G) = 3/2$.
By Theorem \ref{t:BRS}, there exist $T=\{x,y,z\} \in \mathbb{T}_1$ and $p \in [xy]$ with $d_G(p,[xz] \cup [yz])=\delta(G) = 3/2$.
Since we have proved $\diam V(G) \le 2$ and $\diam G \le 3$, we have $d_G(x,y) = 3$, $d_G(p,\{x,y\}) = d_G(p,[xz] \cup [yz])= 3/2$,
$x,y \in J(G) \setminus V(G)$ and $p \in V(G)$.
Hence, $x$ (respectively, $y$) is the midpoint of $u_1u_2\in E(G)$ with $u_1,u_2\in V(G)\setminus K$ and corresponding arcs $H_1,H_2$
(respectively, the midpoint of $w_1w_2\in E(G)$ with $w_1,w_2\in V(G)\setminus K$ and corresponding arcs $J_1,J_2$).
Note that each arc $H_1,H_2,J_1,J_2$ intersects $I_1 \cap I_2$ and it is different from $I_1$ and $I_2$.
Since $G$ is a proper circular-arc graph, we have that both $H_1$ and $H_2$ contain the same connected component $\Lambda$ of $I_1 \cap I_2$;
also, both $J_1$ and $J_2$ contain the other connected component $\Lambda'$ of $I_1 \cap I_2$.
Denote by $I$ the corresponding arc to $p$.
Since $G$ is a proper circular-arc graph, we have that $I$ contains either $\Lambda$ or $\Lambda'$.
Assume that $I$ contains $\Lambda$ (if $I$ contains $\Lambda'$, then the argument is similar).
Thus, $d_G(p,u_1) = d_G(p,u_2) = 1$.
Without loss of generality we can assume that $u_1 \in [xy]$.
Therefore, $u_2 \in [xz] \cup [yz]$ and we conclude
$3/2 = d_G(p,[xz] \cup [yz]) \le d_G(p,u_2) = 1$,
a contradiction.
Hence, $\d(G) < 3/2$ and Theorem \ref{t:BRS} gives $\d(G)\le 5/4$.

%
%
%
%
%

Finally, we show that the proper circular-arc graph $\G$ (with $\varrho(G) = 2$) corresponding to the arcs
$$
[e^{0i},e^{\pi i}] \cup [e^{\pi i},e^{2\pi i}]
\cup [e^{-\pi i/8},e^{\pi i/4}] \cup [e^{-\pi i/4},e^{\pi i/8}]\cup [e^{\pi i-\pi i/8},e^{\pi i+\pi i/4}] \cup [e^{\pi i-\pi i/4},e^{\pi i+\pi i/8}],
$$
satisfies $\d(\G) = 5/4$.
Let $x$ (respectively, $y$) be the midpoint of the edge in $\G$ with endpoints $[e^{-\pi i/8},e^{\pi i/4}]$ and $[e^{-\pi i/4},e^{\pi i/8}]$
(respectively, $[e^{\pi i-\pi i/8},e^{\pi i+\pi i/4}]$ and $[e^{\pi i-\pi i/4},e^{\pi i+\pi i/8}]$).
We have $d_{\G}(x,y)=3$.
One can check that there are two geodesics $\g_1$ and $\g_2$ such that the midpoint $q$ of $\g_1$ is a vertex of $\G$ and $d_{\G}(q,\g_2)=1$.
If $p$ is a point in $\g_1$ with $d_{\G}(p,q)=1/4$, then $d_{\G}(p,\g_2)=5/4$.
If we consider the geodesic bigon $\{\g_1,\g_2\}$, then $5/4 = d_{\G}(p,\g_2)\le \d(\G) \le 5/4$, and we conclude $\d(\G) = 5/4$.
\end{proof}

Note that Theorem \ref{t:l4} below gives a sufficient condition in order to attain the lower bound of $\d(G)$ in Theorem \ref{t:main}.
This sufficient condition is, in fact, a characterization when $3 \le \varrho(G)\le 4$.

\smallskip

Next, we are going to characterize the circular-arc graphs with the two smallest possible values for the hyperbolicity constant: $0$ and $3/4$.

\smallskip

We say that a circular-arc graph $G$ has the $0$-\emph{property} if we have either:

$(1)$ $G$ is an interval graph with the $0$-intersection property.

$(2)$ $\varrho(G)=1$ and given two corresponding arcs $I,J$ to vertices in $G$ with $I,J\neq \SS$, we have $I \cap J = \emptyset$.

$(3)$ $\varrho(G)=2$ and there exist two corresponding arcs $I_1,I_2$ to vertices in $G$ with $I_1 \cup I_2 = \SS$ such any other corresponding arc
to some vertex in $G$ intersects just one of the arcs $I_1,I_2,$
and if $G_j$ is the interval graph corresponding to the arcs intersecting $I_j$ then $G_j$ has the $0$-intersection property for $j=1,2$.

\begin{proposition} \label{t:0}
A circular-arc graph $G$ satisfies $\d(G) = 0$ if and only if $G$ has the $0$-property.
\end{proposition}

\begin{proof}
If $G$ is an interval graph, then Theorem \ref{034} gives the result.
Assume now that $G$ is a NI circular-arc graph.

If $G$ satisfies $(2)$ in the definition of $0$-property, then $G$ is a tree (in fact, it is a star graph) and we have $\d(G) = 0$.

If $G$ satisfies $(3)$ in the definition of $0$-property, then $G$ is a tree and we have $\d(G) = 0$.

Assume that $\d(G) = 0$.
Theorem \ref{t:main} gives that $\varrho(G) \le 2$.

Assume that $\varrho(G) = 1$.
Seeking for a contradiction assume that there exist two corresponding arcs $I,J$ to vertices in $G$ with $I,J\neq \SS$ and $I \cap J \neq \emptyset$.
Therefore, there exists a cycle with length three corresponding to the arcs $I,J, \SS$,
and Lemma \ref{p:c3} gives $0= \d(G) \ge 3/4$,
a contradiction.
Thus, we have $I \cap J = \emptyset$ and $G$ has the $0$-property.

Assume that $\varrho(G) = 2$.
Thus, there exist two corresponding arcs $I_1,I_2$ to vertices in $G$ with $I_1 \cup I_2 = \SS$.
Seeking for a contradiction assume that there exists a corresponding arc $I$
to some vertex in $G$ intersecting both arcs $I_1$ and $I_2$.
Therefore, there is a cycle of length $3$ in $G$ corresponding to $I,I_1,I_2,$ and we have $\d(G) \ge 3/4$ by Lemma \ref{p:c3},
which is a contradiction.
So, any other corresponding arc
to some vertex in $G$ intersects just one of the arcs $I_1,I_2$.
Let $G_j$ be the interval graph corresponding to the arcs intersecting $I_j$ for $j=1,2$.
Since $\d(G) = 0$, Proposition \ref{t:Tdec} gives that $\d(G_1) = \d(G_2) = 0$.
Thus, Theorem \ref{034} gives that $G_1$ and $G_2$ have the $0$-intersection property.
\end{proof}

We say that a circular-arc graph $G$ has the $(3/4)$-\emph{property} if we have either:

$(1)$ $G$ is an interval graph with the $(3/4)$-intersection property.

$(2)$ $\varrho(G)=1$,
there exist two corresponding arcs $I',I''\neq \SS$ to vertices in $G$ with $I'\cap I'' \neq \emptyset$,
and for every three corresponding arcs $I,J,K\neq \SS$ to vertices in $G$, we have either $I\cap J = \emptyset$ or $I\cap K = \emptyset$.

$(3)$ $\varrho(G)=2$ and there exist two corresponding arcs $I_1,I_2$ to vertices in $G$ with $I_1 \cup I_2 = \SS$ such any other corresponding arc
to some vertex in $G$ intersects just one of the arcs $I_1,I_2,$
and if $G_j$ is the interval graph corresponding to the arcs intersecting $I_j$ then
$G_1$ has the $(3/4)$-intersection property and
$G_2$ has either the $0$- or the $(3/4)$-intersection property.

$(4)$ $\varrho(G)=2$ and there exist three corresponding arcs $I,I_1,I_2$ to vertices in $G$ with $I_1 \cup I_2 = \SS$
and $I\cap I_j \neq \emptyset$ for $j=1,2,$ such that any other arc
corresponding to some vertex in $G$ intersects just one of the arcs $I_1,I_2$ and does not intersect $I$,
and if $G_j$ is the interval graph corresponding to the arcs intersecting $I_j$ except for $I$ then $G_j$ has either the $0$- or the $(3/4)$-intersection property for each  $j=1,2$.

$(5)$ $\varrho(G)=3$ and there exist three corresponding arcs $I_1,I_2,I_3$ to vertices in $G$ with $I_1 \cup I_2 \cup I_3 = \SS$
such any other corresponding arc
to some vertex in $G$ intersects just one of the arcs $I_1,I_2,I_3,$
and if $G_j$ is the interval graph corresponding to the arcs intersecting $I_j$ then $G_j$ has either the $0$- or the $(3/4)$-intersection property for each  $j=1,2,3$.

\begin{proposition} \label{t:34}
A circular-arc graph $G$ satisfies $\d(G) = 3/4$ if and only if $G$ has the $(3/4)$-property.
\end{proposition}

\begin{proof}
If $G$ is an interval graph, then Theorem \ref{034} gives the result.
Assume now that $G$ is a NI circular-arc graph.

If $G$ satisfies either $(2)$, $(3)$, $(4)$ or $(5)$ in the definition of $(3/4)$-property,
then Theorems \ref{t:04} and \ref{034} and Proposition \ref{t:Tdec} give that $\d(G) = 3/4$.

Assume that $\d(G) = 3/4$.
Theorem \ref{t:main} gives $\varrho(G) \le 3$.

Assume that $\varrho(G) = 3$.
Thus, there exist three corresponding arcs $I_1,I_2,I_3$ to vertices in $G$ with $I_1 \cup I_2 \cup I_3 = \SS$
Seeking for a contradiction assume that there exists a corresponding arc $I$
to some vertex in $G$ intersecting at least two arcs in $\{I_1,I_2,I_3\}$.
Therefore, there is a cycle of length $4$ in $G$ corresponding to $I,I_1,I_2,I_3$ and we have $\d(G) \ge 1$ by Lemma \ref{p:c3}, which is a contradiction.
So, any other corresponding arc
to some vertex in $G$ intersects just one of the arcs $I_1,I_2,I_3$.
Let $G_j$ be the interval graph corresponding to the arcs intersecting $I_j$ for $j=1,2,3$.
Let us denote by $G_0$ the subgraph of $G$ induced by the corresponding vertices to $I_1,I_2,I_3$ ($G_0$ is a cycle graph with three vertices).
Note that $\{G_0,G_1,G_2,G_3\}$ is a T-decomposition of $G$.
Since $\d(G) = 3/4$, Proposition \ref{t:Tdec} gives that $\d(G_j) \le 3/4$ for $j=1,2,3$.
Thus, Theorem \ref{034} gives that $G_j$ has either the $0$- or the $(3/4)$-intersection property for each  $j=1,2,3,$ and we obtain condition $(5)$.

Assume that $\varrho(G) = 2$.
Thus, there exist two corresponding arcs $I_1,I_2$ to vertices in $G$ with $I_1 \cup I_2 = \SS$.

Assume that any other corresponding arc
to some vertex in $G$ intersects just one of the arcs $I_1,I_2$.
Let $G_j$ be the interval graph corresponding to the arcs intersecting $I_j$ for $j=1,2,$
and let $G_0$ be the subgraph of $G$ induced by the corresponding vertices to $I_1,I_2$ ($G_0$ has just an edge).
Since $\{G_0,G_1,G_2\}$ is a T-decomposition of $G$, Proposition \ref{t:Tdec} gives
$$
\frac34
= \d(G)
= \max \big\{ \d(G_0), \d(G_1), \d(G_2)\big\}
= \max \big\{ \d(G_1), \d(G_2)\big\}.
$$
Hence, a subgraph, say $G_1$, has hyperbolicity constant $3/4$ and $\d(G_2)\le 3/4$.
Thus, Theorems \ref{t:04} and \ref{034} give that $G_1$ has the $(3/4)$-intersection property and
$G_2$ has either the $0$- or the $(3/4)$-intersection property, and we obtain condition $(3)$.


Assume that there exist corresponding arcs $I,I_1,I_2$ to vertices in $G$ with $I\cap I_j \neq \emptyset$ for $j=1,2$.
Seeking for a contradiction assume that there exists another corresponding arc $J$
to some vertex in $G$ intersecting both arcs $I_1$ and $I_2$.
Hence, there is a cycle of length four in $G$ corresponding to $I,I_1,I_2,J$ and we have $\d(G) \ge 1$ by Lemma \ref{p:c3}, which is a contradiction.
Thus, any other corresponding arc
to some vertex in $G$ intersects just one of the arcs $I_1,I_2$.
A similar argument gives that any other corresponding arc to some vertex in $G$ does not intersect $I$.
Let $G_j$ be the interval graph corresponding to the arcs intersecting $I_j$ except for $I$.
Let us denote by $G_0$ the subgraph of $G$ induced by the corresponding vertices to $I,I_1,I_2$ ($G_0$ is a cycle graph with three vertices).
Since $\{G_0,G_1,G_2\}$ is a T-decomposition of $G$, Proposition \ref{t:Tdec} gives
$$
\frac34
= \d(G)
= \max \big\{ \d(G_0), \d(G_1), \d(G_2)\big\}
= \max \Big\{ \frac34\,, \d(G_1), \d(G_2)\Big\}.
$$
This equation holds if and only if $\d(G_j)\le 3/4$ for $j=1,2$.
Thus, Theorems \ref{t:04} and \ref{034} give that $G_j$ has either the $0$- or the $(3/4)$-intersection property for each $j=1,2,$ and we obtain condition $(4)$.

Finally, assume that $\varrho(G) = 1$.

Seeking for a contradiction assume that
for every two corresponding arcs $I',I''\neq \SS$ to vertices in $G$, we have $I'\cap I'' = \emptyset$.
Thus, $G$ is a star graph and $\d(G)=0$, a contradiction.
Hence, there exist two arcs $I',I''\neq \SS$
with $I'\cap I'' \neq \emptyset$.

Seeking for a contradiction assume that
there exist three corresponding arcs $I,J,K\neq \SS$ to vertices in $G$
with $I\cap J \neq \emptyset$ and $I\cap K \neq \emptyset$.
Therefore, there exists a cycle with length four corresponding to the arcs $I,J,K, \SS$,
and Lemma \ref{p:c3} gives $\d(G) \ge 1$,
a contradiction.
Thus, for every three arcs $I,J,K\neq \SS$
we have either $I\cap J = \emptyset$ or $I\cap K = \emptyset$,
and we obtain condition $(2)$.
\end{proof}

We say that a circular-arc graph $G$ with $\varrho(G) \ge 3$ has the $\varrho(G)$-\emph{property} if there exist $\varrho(G)$ corresponding arcs $I_1,\dots,I_{\varrho(G)}$
to vertices in $G$ with $I_1 \cup \cdots \cup I_{\varrho(G)} = \SS$ such any other corresponding arc
to some vertex in $G$ intersects just one of the arcs $I_1,\dots,I_{\varrho(G)},$
and if $G_j$ is the interval graph corresponding to the arcs intersecting $I_j$ for $1\le j \le \varrho(G)$ and $\varrho(G) \le 5$ then:

$(1)$ $G_j$ has either the $0$- or $(3/4)$-intersection property for $1\le j \le \varrho(G)$ if $\varrho(G)=3$.

$(2)$ $G_j$ has either the $0$-, $(3/4)$- or $1$-intersection property for $1\le j \le \varrho(G)$ if $\varrho(G)=4$.

$(3)$ $G_j$ does not have the $(3/2)$-intersection property for $1\le j \le \varrho(G)$ if $\varrho(G)=5$.

\smallskip

The next result gives a sufficient condition in order to attain the lower bound of $\d(G)$ in Theorem \ref{t:main}.
We also prove that this sufficient condition is, in fact, a characterization when $3 \le \varrho(G)\le 4$.

\begin{theorem} \label{t:l4}
Let $G$ be a circular-arc graph with $\varrho(G) \ge 3$.
If $G$ has the $\varrho(G)$-property, then $G$ satisfies $\d(G) = \varrho(G)/4$.
Furthermore, if $\d(G) = \varrho(G)/4$ with $3 \le \varrho(G)\le 4$, then $G$ has the $\varrho(G)$-property.
\end{theorem}

\begin{proof}
Assume first that $G$ has the $\varrho(G)$-property.
Let us denote by $G_0$ the subgraph of $G$ induced by the corresponding vertices to $I_1,\dots,I_{\varrho(G)}$ ($G_0$ is a cycle graph with $\varrho(G)$ vertices).
Since $\{G_0,G_1,\dots,G_{\varrho(G)}\}$ is a T-decomposition of $G$, Theorem \ref{t:main}, Proposition \ref{t:Tdec} and Lemma \ref{cyclegarph} give
\begin{equation} \label{eq:l4}
\frac{\varrho(G)}4
\le \d(G)
= \max \big\{ \d(G_0), \d(G_1), \dots, \d(G_{\varrho(G)})\big\}
= \max \Big\{ \frac{\varrho(G)}4\,, \d(G_1), \dots, \d(G_{\varrho(G)})\Big\}.
\end{equation}
Since $G_j$ is an interval graph for $1\le j \le \varrho(G)$,
if $\varrho(G) \ge 6$, then Theorem \ref{034} gives
$\d(G_j) \le 3/2 \le \varrho(G)/4$ for $1\le j \le \varrho(G)$.

If $\varrho(G) = 3$, then Theorem \ref{034} gives
$\d(G_j) \le 3/4 = \varrho(G)/4$ for $1\le j \le \varrho(G)$.

If $\varrho(G) = 4$, then Theorem \ref{034} gives
$\d(G_j) \le 1 = \varrho(G)/4$ for $1\le j \le \varrho(G)$.

If $\varrho(G) = 5$, then Theorem \ref{034} gives
$\d(G_j) \le 5/4 = \varrho(G)/4$ for $1\le j \le \varrho(G)$.

These inequalities and \eqref{eq:l4} give $\d(G) = \varrho(G)/4$ in every case.

Assume now that $\d(G) = \varrho(G)/4$ with $3 \le \varrho(G) \le 4$.

Seeking for a contradiction assume that there exists a corresponding arc $I$
to some vertex in $G$ intersecting at least two arcs in $\{I_1,\dots,I_{\varrho(G)} \}$.
Denote by $\{v_1,\dots,v_{\varrho(G)} \}$ their corresponding vertices in $G$, and by $C$ the cycle in $G$ with vertices $\{v_1,\dots,v_{\varrho(G)} \}$.
Let $v_I$ be the corresponding vertex in $G$ to $I$.

If $\varrho(G) =3$, then there is a cycle in $G$ with vertices $\{v_1,v_2,v_3,v_I \}$,
and Lemma \ref{p:c3} gives $\d(G)\ge 1$, a contradiction.

If $\varrho(G) =4$, then we show now that there is a cycle in $G$ with vertices $\{v_1,v_2,v_3,v_4,v_I \}$.
The definition of $\varrho(G)$ gives that $v_I$ is neighbor of at most three vertices in $\{v_1,v_2,v_3,v_4 \}$.
Without loss of generality we can assume that
$v_1v_2, v_2v_3, v_3v_4, v_4v_1, v_1v_I, v_2v_I\in E(G)$ and $v_4v_I \notin E(G)$.
Consider the cycle $g:= v_1v_I \cup v_Iv_2 \cup v_2v_3 \cup v_3v_4 \cup v_4v_1$ in $G$.
Since $L(g) = 5$ and $\deg_g(v_4)=2$,
Theorem \ref{t:delta2} gives $\d(G)\ge 5/4$, a contradiction.

Thus, any corresponding arc to some vertex in $G$ intersects just one of the arcs $\{I_1,\dots,I_{\varrho(G)} \}$.
Let $G_j$ be the interval graph corresponding to the arcs intersecting $I_j$ for $1\le j \le \varrho(G)$.
Let us denote by $G_0$ the subgraph of $G$ induced by the corresponding vertices to $I_1,\dots,I_{\varrho(G)}$ for $1\le j \le \varrho(G)$ ($G_0$ is a cycle graph with $\varrho(G)$ vertices).
Since $\{G_0,G_1,\dots,G_{\varrho(G)}\}$ is a T-decomposition of $G$, Proposition \ref{t:Tdec} and Lemma \ref{cyclegarph} give
$$
\frac{\varrho(G)}4
= \d(G)
= \max \big\{ \d(G_0), \d(G_1), \dots, \d(G_{\varrho(G)})\big\}
= \max \Big\{ \frac{\varrho(G)}4\,, \d(G_1), \dots, \d(G_{\varrho(G)})\Big\}.
$$
This equation holds if and only if $\d(G_j)\le \varrho(G)/4$ for $1\le j \le \varrho(G)$.

If $\varrho(G) = 3$, then Theorem \ref{034} gives
$\d(G_j) \le 3/4 = \varrho(G)/4$ for $1\le j \le \varrho(G)$ if and only if $(1)$ holds.

If $\varrho(G) = 4$, then Theorem \ref{034} gives
$\d(G_j) \le 1 = \varrho(G)/4$ for $1\le j \le \varrho(G)$ if and only if $(2)$ holds.

Hence, $G$ has the $\varrho(G)$-property.
\end{proof}

\begin{example} \label{ex:x}
The second statement in Theorem \ref{t:l4} does not hold for $\varrho \ge 5$, as the following example shows.
Consider the graph $G_\varrho$ obtained from the cycle graph $C_\varrho$ with $\varrho$ vertices and an additional vertex $v_I$
connected by an edge with just three consecutive vertices in $C_\varrho$.
We have that $G_\varrho$ is a circular-arc graph without the $\varrho$-property.
\cite[Theorem 30]{MRSV} gives that the hyperbolicity constant of any graph with $n$ vertices is at most $n/4$.
Hence, $\varrho/4 \le \d(G_\varrho) \le (\varrho+1)/4$.
By using the characterization in \cite[Theorem 30]{MRSV} of the graphs with $n$ vertices and hyperbolicity constant $n/4$,
we obtain $\d(G_\varrho) < (\varrho+1)/4$.
Since $\d(G)$ is a multiple of $1/4$ by Theorem \ref{t:BRS}, we conclude $\d(G_\varrho) = \varrho/4$.
\end{example}

\section{Complement and line graph}

In this section we obtain bounds for the hyperbolicity constant of the complement and line of a circular-arc graph, respectively.
These theorems improve, for circular-arc graphs, the general bounds for the hyperbolicity constant of the complement and line graphs.

\smallskip

Given any graph $G$, we denote by $\overline G$ the complement of $G$, i.e.,
$\overline G$ is the graph with $V(\overline{G})=V(G)$
and $vw \in E(\overline{G})$ if and only if $vw \notin E( G )$.

\smallskip

Now, we are interested in the hyperbolicity of the complement of circular-arc graphs.
Let us start with two technical results.

\begin{lemma} \label{l:15}
Let $G$ be a circular-arc graph with $\varrho(G) \ge 1$.
If two vertices $u$ and $v$ are not neighbors and have two common neighbors $v_1,v_2,$ such that $v_1$ and $v_2$ are not neighbors,
then their corresponding arcs satisfy $I_{u}\cup I_{v}\cup I_{v_1}\cup I_{v_2} = \SS$.
\end{lemma}

\begin{proof}
Since $\varrho(G) \ge 1$,
$I_{u} \cap I_{v} = \emptyset$,
$I_{v_1} \cap I_{v_2} = \emptyset$,
$I_{u} \cap I_{v_j} \neq \emptyset$ and
$I_{v} \cap I_{v_j} \neq \emptyset$ for $j=1,2,$
we have $I_{u} \cup I_{v} \cup I_{v_1} \cup I_{v_2} = \SS$.
\end{proof}

\begin{lemma} \label{l:14}
Let $G$ be a circular-arc graph with $\varrho(G) > 4$.
If two vertices $u$ and $v$ have two common neighbors $v_1,v_2,$ such that $v_1$ and $v_2$ are not neighbors, then $u$ and $v$ are neighbors.
\end{lemma}

\begin{proof}
Seeking for a contradiction, assume that $u$ and $v$ are not neighbors.
Lemma \ref{l:15} gives $I_{u} \cup I_{v} \cup I_{v_1} \cup I_{v_2} = \SS$,
and thus $\varrho(G) \le 4$, a contradiction.
So, $u$ and $v$ are neighbors.
\end{proof}

Recall that a graph is $s$-regular if every vertex has degree $s$, i.e., has $s$ neighbors.
In order to prove Theorem \ref{t:comple} below we need the following surprising result about regular graphs which is interesting by itself.

\begin{theorem} \label{t:regular}
Let $G$ be a $(n-3)$-regular graph with $n\ge 5$ vertices.
Then $\d(G)=1$ if $\overline G$ is a union of cycle graphs with three vertices, and $\d(G)=5/4$ otherwise.
\end{theorem}

\begin{proof}
Assume first that $\overline G$ is a union of cycle graphs with three vertices (thus, $n\ge 6$).
\cite[Lemma 5.7]{HPR} gives that $\d(G)\le 1$.
Since $n\ge 6$, we have $n-3\ge n/2$ and there exists a Hamiltonian cycle with $n\ge 6$ vertices; thus, Lemma \ref{p:c3} gives that $\d(G)\ge 1$.

Assume now that $\overline G$ is not a union of cycle graphs with three vertices.
If $n=5$, then $G$ is a cycle graph with five vertices and $\d(G) = 5/4$.
Assume that $n\ge 6$.
Hence, there exists $v \in V(G)$ such that the connected component of $\overline G$ containing $v$ is not a cycle graph with three vertices.
Let $v_1,v_2$ be the vertices with $v_1v,v_2v \notin E(G)$.
Seeking for a contradiction assume that $v_1v_2 \notin E(G)$.
Thus, the connected component of $\overline G$ containing $v$ is the cycle graph with vertices $v,v_1,v_2,$ a contradiction.
Hence, $v_1v_2 \in E(G)$.
Since $n\ge 6$, we have $2(n-3) \ge n$ and there are at least two common neighbors of $v$ and $v_j$ for each $j=1,2$.
Therefore, there exist two different vertices $v_3,v_4$ with $v_3v,v_4v,v_3v_1,v_4v_2 \in E(G)$,
and we have the cycle $g$ given by $v,v_3,v_1,v_2,v_4,v$ in $G$.
Since $L(g) = 5$ and $\deg_g(v)=2$ (recall that $v_1v,v_2v \notin E(G)$), Theorem \ref{t:delta2} gives $\delta(G) \ge 5/4$.
Finally, Theorem \ref{t:HPR2} gives $\d(G)\le 5/4$.
\end{proof}

Theorem \ref{t:regular} has the following direct consequence.

\begin{corollary} \label{c:regular}
If $G$ is a $(n-3)$-regular graph with $n\ge 5$ vertices and $n$ is not a multiple of $3$, then $\d(G)=5/4$.
\end{corollary}

The following result provides sharp bounds for the hyperbolicity constant of the complement of any circular-arc graph
(even the circular-arc graphs $G$ with $\diam V(G) = 2$).
Note that it improves Theorem \ref{t:comple00} for circular-arc graphs;
recall that the most difficult case in the study of the complement of a graph are the graphs $G$ with $\diam V(G) = 2$
(this is the case if $\varrho(G) = 4$ or $\varrho(G) = 5$),
and that Theorem \ref{t:comple00} does not deal with these graphs.

\begin{theorem} \label{t:comple}
Let $G$ be a circular-arc graph.
If $\varrho(G) = 0$, then $0 \le \d(\overline{G}) \le 2$.
If $\varrho(G) > 4$, then $5/4 \le \d(\overline{G}) \le 3/2$.
If $\varrho(G) = 4$, then $0 \le \d(\overline{G}) \le 7/2$.
Furthermore, the lower bounds are sharp; in particular, they are attained by the cycle graphs for $\varrho(G) \ge 4$.
\end{theorem}

\begin{proof}
If $\varrho(G)=0$, then Theorem \ref{t:comple0} gives the result.

Assume now that $\varrho(G) > 4$.
We are going to prove that $\diam \overline G \le 3$
(note that it is possible to have $\diam V(G) = 2$, and that the inequality $\diam \overline G \le 3$ is stronger than $\diam V(\overline G ) \le 3$).
Seeking for a contradiction assume that $\diam \overline G > 3$.

Assume first that $\diam V(\overline{G}) \ge 4$.
Thus,
there exist $v,w \in V(\overline{G})$ with $d_{\overline G}(v,w)=4$.
Let $v_0=v,v_1,v_2,v_3,v_4=w \in V(\overline{G})$ such that $v_{j-1}v_j\in E(\overline{G})$ for $1 \le j \le 4$.
Therefore, $v_0$ and $v_1$ have two common neighbors $v_3,v_4$ in $G$ with $v_{3}v_4\notin E(G)$,
and Lemma \ref{l:14} gives that $v_0$ and $v_1$ are neighbors in $G$.
This contradicts $v_{0}v_1\in E(\overline{G})$.

Assume that $\diam V(\overline{G}) = 3$.
Thus, there exist $v \in V(\overline{G})$ and a midpoint $x$ of an edge $v_3v_3'$ in $\overline G$ with $d_{\overline G}(v,x)=7/2$.
Hence,
there exist $v_0=v,v_1,v_2 \in V(\overline{G})$ such that $v_{j-1}v_j \in E(\overline{G})$ for $1 \le j \le 3$.
Therefore, $v_3$ and $v_3'$ have two common neighbors $v_0,v_1$ in $G$ with $v_{0}v_1\notin E(G)$,
and Lemma \ref{l:14} gives that $v_3$ and $v_3'$ are neighbors in $G$.
This contradicts $v_{3}v_3'\in E(\overline{G})$.

Hence, $\diam \overline G \le 3$ and Lemma \ref{dddd} gives $\d(\overline{G}) \le 3/2$.

In order to prove the lower bound, consider a cycle $C$ in $G$ given by $v_1,v_2,\dots,v_{\varrho(G)},v_1$ such that the subgraph induced by this vertices is $C$.
Consider the cycle $g$ in $\overline G$ given by $v_3,v_5,v_2,v_4,v_1,v_3$.
Since $L(g) = 5$ and $\deg_g(v_3)=2$ (recall that $v_3v_2,v_3v_4 \notin E(\overline G )$), Theorem \ref{t:delta2} gives $\d(\overline G ) \ge 5/4$.

Consider now the cycle graph with $\varrho$ vertices $C_\varrho$.
Since $\varrho > 4$, $\overline C_\varrho$ is $(\varrho-3)$-regular and its complement is $C_\varrho$,
Theorem \ref{t:regular} gives $\d\big(\,\overline C_\varrho \big) = 5/4$ and the bound is attained.

\smallskip

Finally, assume that $\varrho(G) = 4$
and consider a geodesic triangle $T=\{x,y,z\}$ in $\overline{G}$ and $p\in [xy]$.
By Theorem \ref{t:BRS}, we can assume that $x,y,z \in J(\overline{G})$.
If $d_{\overline G}(x,y) \le 4$, then $d_{\overline G}(p,[xz]\cup [yz]) \le 2 < 7/2$.
Assume that $d_{\overline G}(x,y) > 4$.
Since $x,y,z \in J(\overline{G})$, we have $d_{\overline G}(x,y) \ge 9/2$.
Thus, there exist $u,v \in V(\overline{G}) \cap [xy]$ with $d_{\overline G}(u,v) =4$
and vertices $v_0=u,v_1,v_2,v_3,v_4=v \in V(\overline{G}) \cap [xy]$
such that $v_{j-1}v_j\in E(\overline{G})$ for $1 \le j \le 4$ and $d_{\overline G}(p,v_2) \le 1/2$.
Therefore, $v_0$ and $v_1$ have two common neighbors $v_3,v_4$ in $G$, and $v_{0}v_1,v_{3}v_4\notin E(G)$.
Hence, Lemma \ref{l:15} gives that their corresponding arcs satisfy $I_{v_0}\cup I_{v_1}\cup I_{v_3}\cup I_{v_4} = \SS$.
Seeking for a contradiction assume that there exists a vertex $w_0 \in V(\overline{G})$ with corresponding arc $I_{w_0}$
such that $I_{w_0}\cap I_{v_j}\neq \emptyset$ for $j=0,1,3,4$.
Since $I_{v_0}\cup I_{v_1}\cup I_{v_3}\cup I_{v_4} = \SS$, there exist $i,j \in \{0,1,3,4\}$ with
$I_{w_0}\cup I_{v_i}\cup I_{v_j} = \SS$.
This contradicts $\varrho(G) = 4$, and so every vertex $w_0 \in V(\overline{G})$ has at most three neighbors in $\{v_{0},v_1,v_{3}v_4\}$ in $G$.
Thus, given any vertex $w_0 \in V(\overline{G}) \cap ([xz]\cup [yz])$, there exists $k \in \{0,1,3,4\}$ with
$w_{0}v_k\notin E(G)$, and
$$
d_{\overline G}(p,[xz]\cup [yz])
\le d_{\overline G}(p,w_0)
\le d_{\overline G}(p,v_2) + d_{\overline G}(v_2,v_k) + d_{\overline G}(v_k,w_0)
\le \frac12 +2 + 1
= \frac72 \,.
$$
So, $\d(\overline{G}) \le 7/2$.

The lower bound $\d(\overline{G}) \ge 0$ trivially holds.
If we consider the cycle graph $G=C_4$, then $\overline{G}$ is the union of two disjoint edges and
$\d(\overline{G}) = 0$.
Hence, the lower bound is attained.
\end{proof}

\begin{remark}
Note that Theorems \ref{t:BRS} and \ref{t:comple}
give that if $G$ is a circular-arc graph with $\varrho(G) > 4$, then we have either $\d(\overline{G}) = 5/4$ or $\d(\overline{G}) = 3/2$.
\end{remark}


Theorems \ref{t:main} and \ref{t:comple} have the following consequence.

\begin{corollary} \label{c:comple}
If $G$ is a circular-arc graph with $\varrho(G) \ge 7$, then $\d(\overline{G}) < \d(G)$.
\end{corollary}

In 1956, Nordhaus and Gaddum gave lower and upper bounds on the sum and the product
of the chromatic number of a graph and its complement in \cite{NG}.
Since then, relations of a similar type have been proposed for many other graph invariants,
in several hundred papers (see, e.g., \cite{AH}).

Also, Theorems \ref{t:main} and \ref{t:comple} provide some Nordhaus-Gaddum type results.

\begin{corollary} \label{t:prodsum}
If $G$ is a circular-arc graph, then
$$
\begin{aligned}
\frac{5\varrho(G)}{16}
\le \delta(G) \delta(\overline{G}) \le
\frac{3\varrho(G)}{8} + \frac{9}{4}\,,
\qquad
& \frac{\varrho(G)+5}{4}
\le \delta(G)+ \delta(\overline{G}) \le
\frac{\varrho(G)}{4} + 3 ,
\qquad \text{if }\;  \varrho(G) > 4,
\\
0
\le \delta(G) \delta(\overline{G}) \le
\frac{7\varrho(G)}{8} + \frac{21}{4}\,,
\qquad
& \frac{\varrho(G)}{4}
\le \delta(G)+ \delta(\overline{G}) \le
\frac{\varrho(G)}{4} + 5 ,
\qquad \text{if }\;  \varrho(G) = 4,
\\
0
\le \delta(G) \delta(\overline{G}) \le
\frac{\varrho(G)}{2} + 3 ,
\qquad
& \frac{\varrho(G)}{4}
\le \delta(G)+ \delta(\overline{G}) \le
\frac{\varrho(G)}{4} + \frac{7}{2}\,,
\qquad \text{if }\;  \varrho(G) = 0.
\end{aligned}
$$
\end{corollary}

If $G$ is a graph with edges $E(G)=\{e_i\}_{i\in\mathcal{I}}$, the \emph{line graph} $\L(G)$ of $G$ is a graph which has a vertex $v_{e_i}\in V(\L(G))$
for each edge $e_i$ of $G$, and an edge joining $v_{e_i}$ and $v_{e_j}$ when $e_i \cap e_j \neq \emptyset$.
The line graph of $G$ is interesting in the theory of geometric graphs, since it is the intersection graph of $E(G)$.

A graph is \emph{chordal} if all cycles of length at least four have a chord, which is an edge that is not part of the cycle but connects two vertices of the cycle
(i.e., it does not have induced cycles of length greater than three).

The following result appears in \cite[Lemma 2.2]{BHB1}.

\begin{lemma} \label{l:BHB1}
Suppose that $G$ is chordal, and that $x_1, x_2,\dots,x_n, x_1$ is a cycle in $G$, where $n \ge 4$.
If $d(x_1,x_3) = 2$, then there exists $i \in \{4, 5,\dots ,n\}$ such that $x_ix_2 \in E(G)$.
\end{lemma}

We want to prove a similar result for $\L(G)$.
In order to do it we need some background.

Given $v_e \in V(\L(G))$, let us define $h(v_e)$ as the midpoint of the edge $e\in E(G)$ and $H(v_e)=e$.
Thus, $h$ and $H$ are maps with $h: V(\L(G)) \rightarrow G$ and $H: V(\L(G)) \rightarrow E(G)$.

\cite[Remark 3.3]{CRS} gives that the map $h$ is an isometry:

\begin{lemma} \label{l:CRS}
For every $x,y \in V (\L(G))$, we have
$$
d_{\L(G)}(x, y) = d_G(h(x), h(y)).
$$
\end{lemma}

\begin{lemma} \label{l:BHB2}
Suppose that $G$ is chordal, and that $u_1, u_2,\dots,u_n, u_1$ is a cycle in $\L(G)$, where $n \ge 6$.
If $d_{\L(G)}(u_1,u_4) = 3$, then there exists $i \in \{4, 5,\dots ,n\}$ and $u\in V(\L(G))$ such that $u_2u,u_3u,u_{i}u,u_{i+1}u\in E(\L(G))$, where $u_{n+1}=u_1$.
\end{lemma}

\begin{proof}
Denote by $C$ the cycle $u_1, u_2,\dots,u_n, u_1$ and by $C_0$ its corresponding cycle in $G$.
Lemma \ref{l:CRS} gives $d_G(h(u_1), h(u_4))=d_{\L(G)}(u_1,u_4) = 3$.
Thus, the vertices $H(u_1)\cap H(u_2)$ and $H(u_3)\cap H(u_4)$ in $C_0$ satisfy
$d_G(H(u_1)\cap H(u_2),H(u_3)\cap H(u_4))= 2$, and Lemma \ref{l:BHB1} gives that there exists
$i \in \{4, 5,\dots ,n\}$ with
$d_G(H(u_2)\cap H(u_3),H(u_i)\cap H(u_{i+1}))= 1$.
If we denote by $u$ the corresponding vertex in $\L(G)$ to the edge in $G$ with endpoints $H(u_2)\cap H(u_3)$ and $H(u_i)\cap H(u_{i+1})$,
then $u_2u,u_3u,u_{i}u,u_{i+1}u\in E(\L(G))$.
\end{proof}

The following result in \cite[Corollary 3.12]{CRS} relates the hyperbolicity constants of $G$ and $\L(G)$.

\begin{theorem} \label{c:IneqLineGraphk}
For any graph $G$ we have
$$
\d(G) \le \d(\L(G)) \le 5 \d(G) + \frac{5}2\,.
$$
\end{theorem}

%
%

If we consider the four-point definition of hyperbolicity, another usual definition, and we denote by $\d'$ the sharp constant for this definition,
we have the following result in \cite[Theorem 6]{CoDu}.

\begin{theorem} \label{c:IneqLineGraphCoDu}
For any graph $G$ we have
$$
\d'(G) -1\le \d'(\L(G)) \le \d'(G) + 1.
$$
\end{theorem}

The upper bound in Theorem \ref{c:IneqLineGraphCoDu} allows to improve the upper bound in Theorem \ref{c:IneqLineGraphk} to $\d(\L(G)) \le 3 \d(G) + c$ for some constant $c$
(unfortunately, the  constant $c$ is greater than $1$ since the four-point definition considers as points just the vertices of the graph, and it is necessary an additional constant in order to deal with continuous triangles).

\smallskip

In order to study the line of circular-arc graphs (see Theorem \ref{t:line} below) we need the following result about the line of chordal graphs which
is interesting by itself, and improves the upper bound $\d(\L(G)) \le 5 \d(G) + 5/2$ in Theorem \ref{c:IneqLineGraphk} (and even the bound $\d(\L(G)) \le 3 \d(G) + c$) for chordal graphs
(since $\d(G) \le 3/2$ for every chordal graph $G$, see \cite{BHB1}).

\begin{theorem} \label{t:chordal}
If $G$ is a chordal graph, then
$$
\d(\L(G))\le \frac52\, .
$$
\end{theorem}

\begin{proof}
Let us consider a geodesic triangle $T=\{x,y,z\}$ in $\L(G)$ and $p\in [xy]$.
By Theorem \ref{t:BRS}, we can assume that $T$ is a cycle.

We are going to prove $d_{\L(G)}(p,[xz]\cup[yz]) \le 3$.
Without loss of generality we can assume that $d_{\L(G)}(p,[xz]\cup[yz]) \ge 2$.
If we denote the vertices of the cycle $T$ by $u_1, u_2,\dots,u_n, u_1,$ then we can assume that $p \in u_2u_3$.
Since $d_{\L(G)}(p,\{x,y\}) \ge d_{\L(G)}(p,[xz]\cup[yz]) \ge 2$, we have $u_1u_2 \cup u_2u_3 \cup u_3u_4 \subset [xy]$
and so, $d_{\L(G)}(u_1,u_4) = 3$.
By Lemma \ref{l:BHB2}, there exist $j \in \{2,3\}$, $k \in \{4, 5,\dots ,n\}$ and $u\in V(\L(G))$ such that $u_2u,u_3u,u_{k}u \in E(\L(G))$ and $d_T(u_j,u_{k})\ge 3$.
Since $u_j \in [xy]$, $d_{\L(G)}(u_j,u_{k})=2$ and $d_T(u_j,u_{k})\ge 3$, we have
$u_k \in [xz]\cup[yz]$ and
$$
d_{\L(G)}(p,[xz]\cup[yz])
\le d_{\L(G)}(p,u_k)
\le d_{\L(G)}(p,\{u_2,u_3\}) + d_{\L(G)}(\{u_2,u_3\},u) + d_{\L(G)}(u,u_k)
\le \frac52\, .
$$
Hence, $\d(\L(G)) \le 5/2$.
\end{proof}

The following result improves the upper bound in Theorem \ref{c:IneqLineGraphk} for circular-arc graphs.

\begin{theorem} \label{t:line}
Let $G$ be a circular-arc graph.
If $\varrho(G) \ge 3$, then
$$
\frac14\, \varrho(G) \le \d(\L(G)) \le \frac12 \Big\lfloor \frac12\, \varrho(G) \Big\rfloor + \, \frac52\, .
$$
If $\varrho(G) = 0,2,$ then
$$
0 \le \d(\L(G))\le \frac52\, .
$$
If $\varrho(G) = 1$, then
$$
0 \le \d(\L(G))\le 2 .
$$
\end{theorem}

\begin{proof}
Theorems \ref{t:main} and \ref{c:IneqLineGraphk} give the lower bounds.

If $\varrho(G)=0$, then Theorem \ref{t:chordal} gives the upper bound, since every interval graph is chordal.

Assume that $\varrho(G)>0$ and let us prove the upper bounds of $\d(\L(G))$.
Fix any set of vertices $K=\{v_1,\dots,v_{\varrho(G)}\}$ and corresponding arcs $\{I_1,\dots,I_{\varrho(G)}\}$ with $I_1\cup \cdots \cup I_{\varrho(G)} = \SS$.

Assume first $\varrho(G) \ge 3$, and denote by $C$ the cycle in $G$ with $V(C)=K$ and by $C'$ the corresponding cycle in $\L(G)$ to $C$.
Since every vertex in $G$ is at distance at most $1$ from $C$, every vertex in $\L(G)$ is at distance at most $2$ from $C'$, and
$$
\begin{aligned}
\diam V(\L(G))
& \le 2  + \diam V(C') + 2 = \Big\lfloor \frac12\, \varrho(G) \Big\rfloor + 4,
\\
\diam \L(G)
& \le \Big\lfloor \frac12\, \varrho(G) \Big\rfloor + 5,
\end{aligned}
$$
and Lemma \ref{dddd} gives the upper bound.

If $\varrho(G) = 2$, then the previous argument gives the desired upper bound, by taking $v_{v_1v_2}$ (with diameter zero) instead of $C'$.

If $\varrho(G) = 1$, then every vertex in $G$ is a neighbor of $v_1$ and the set of edges in $G$ incident on $v_1$ corresponds with a complete graph in $\L(G)$.
Hence, $\diam V(\L(G)) \le 3$, $\diam \L(G) \le 4$ and $\d(\L(G)) \le 2$.
\end{proof}

\end{document}